\documentclass[a4paper,11pt,twoside,reqno]{amsart}

\usepackage[utf8]{inputenc}
\usepackage{amssymb,amsthm,amsmath}
\usepackage[margin=1in]{geometry}
\usepackage[plainpages=false,pdfpagelabels=true]{hyperref}
\usepackage{caption}

\newtheorem{Satz}{Theorem}[section]
\newtheorem{Prop}[Satz]{Proposition}

\newtheorem{Thm}[Satz]{Theorem}
\newtheorem{Cor}[Satz]{Corollary}
\theoremstyle{definition}

\allowdisplaybreaks[1]
\numberwithin{equation}{section}
\renewcommand{\epsilon}{\varepsilon}

\newcommand{\Q}{\ensuremath{\mathbb{Q}}}

\newcommand{\Z}{\ensuremath{\mathbb{Z}}}
\newcommand{\s}{\ensuremath{\mathbb{S}}}
\newcommand{\Null}{\operatorname{Nullity}}
\newcommand{\MWrank}{\operatorname{rank}}

\title{The Nullity of a Family of Proper Biharmonic Maps via Elliptic Curves}

\author{Anna Siffert}
\address{Department of Mathematics\\
University of M\"unster\\
48149 M\"unster, Germany}
\email{asiffert@uni-muenster.de}


\begin{document}

\begin{abstract}
We prove a conjecture of Montaldo, Oniciuc and Ratto concerning the
nullity of a family of proper biharmonic maps from the flat two-torus to the round two-sphere. The proof reveals an unexpected connection between spectral geometry and arithmetic geometry. We show that the vanishing of a mixed Fourier eigenvalue produces a rational point on an explicitly defined affine quartic. By constructing an explicit polynomial isomorphism with an elliptic curve over $\Q$, the problem is reduced to the determination of a Mordell--Weil group. This yields a complete description of the rational points on the spectral curve and shows that none satisfies the positivity conditions required for a mixed Fourier mode. 
As a consequence, the mixed eigenvalues never vanish, confirming the Montaldo--Oniciuc--Ratto conjecture and proving that the nullity of every map in the family is equal to $5$.
\end{abstract}

\maketitle

\section{Introduction}
Proper biharmonic maps constitute one of the most important classes of
critical points of higher-order variational problems in differential
geometry. Let
\[
\varphi:(M,g)\longrightarrow (N,h)
\]
be a smooth map between Riemannian manifolds. Its \emph{bienergy} is
defined by
\[
E_2(\varphi)
=
\frac12\int_M |\tau(\varphi)|^2\,v_g,
\]
where \(\tau(\varphi)\) denotes the tension field of \(\varphi\). A map
is called \emph{biharmonic} if it is a critical point of the bienergy
functional, or equivalently if its bitension field vanishes. Since every
harmonic map satisfies \(\tau(\varphi)=0\), harmonic maps are
automatically biharmonic. A biharmonic map that is not harmonic is called
\emph{proper biharmonic}.

\smallskip

Biharmonic maps thus provide a natural fourth-order generalization of harmonic
maps and have been studied extensively over the last few decades; see,
for example,
\cite{BrandingMontaldoOniciucRatto2020,EellsLemaire1983,Jiang1986,LoubeauOniciuc2005,Urakawa2015}
and the references therein. Besides their intrinsic geometric interest,
they furnish important examples of nonlinear fourth-order elliptic
partial differential equations whose stability properties exhibit a rich
interaction between analysis, geometry and topology.

\smallskip

As for every variational problem, the second variation is fundamental for
the study of stability. For a biharmonic map, the Hessian of the
bienergy is represented by a formally self-adjoint fourthj-order elliptic
operator, called the \emph{Jacobi operator}. Its
spectrum governs the infinitesimal stability of the map. In particular,
its kernel consists of infinitesimal biharmonic deformations, and its
dimension is called the \emph{nullity} of the biharmonic map.

\smallskip

Montaldo, Oniciuc and Ratto studied the distinguished family
\[
\varphi_k:T^2\longrightarrow\mathbb S^2,
\qquad k\in\mathbb N_{>0},
\]
of proper biharmonic maps introduced in \cite{MOR}. Exploiting the
symmetries of the flat torus, they obtained an explicit Fourier
decomposition of the corresponding iterated Jacobi operator. As a
consequence, the computation of the nullity reduces to the analysis of a
family of explicitly computable Fourier eigenvalues. More precisely, they
proved that
\begin{equation}
\label{eq:nullity-MOR}
\Null(\varphi_k)=5+4g(k),
\end{equation}
where \(g(k)\) counts the mixed Fourier modes for which a certain
explicit eigenvalue vanishes. They conjectured that
\[
g(k)=0
\]
for every positive integer \(k\), which would imply that
\(\Null(\varphi_k)=5\) for all \(k\).

\smallskip

The aim of this paper is to prove this conjecture. The key
observation is that the mixed-mode equation admits a natural
interpretation in arithmetic geometry. After dividing by the appropriate
power of \(k\), every hypothetical mixed zero mode determines a rational
point on an explicitly defined affine quartic. A simple change of
variables reduces this quartic to a genus-one equation, and an explicit
polynomial transformation identifies it with the affine part of the
elliptic curve
\[
E:\qquad Y^2=X^3+9X^2-16X+64.
\]
The arithmetic ingredients are standard, see e.g. \cite{SilvermanAEC}. The Mordell--Weil theorem implies that
\(E(\Q)\) is finitely generated \cite{SilvermanAEC,Mordell1922}. A
certified descent computation gives rank zero, and Mazur's theorem
restricts the rational torsion subgroup \cite{MazurTorsion}. In the
present case one obtains
\[
E(\Q)\cong\Z/7\Z.
\]
The explicit inverse polynomial map then gives all rational points on the
quartic. None has both coordinates strictly positive, so none can arise
from a mixed Fourier mode.

\smallskip 

Our main result is the following:
\begin{Thm}
\label{thm:main}
For every positive integer \(k\), the map
\[
\varphi_k:T^2\longrightarrow\s^2
\]
admits no mixed Fourier zero modes. Consequently,
\[
g(k)=0
\qquad\text{and}\qquad
\Null(\varphi_k)=5.
\]
\end{Thm}
This settles the conjecture of Montaldo, Oniciuc and Ratto in affirmative.
The proof separates naturally into a spectral reduction and an arithmetic
classification. 

\medskip

\noindent\textbf{Organization:}
Section~$2$ derives the quartic associated with the mixed
eigenvalue equation. Section~$3$ identifies this quartic with the affine
part of an elliptic curve. Section~$4$ determines the Mordell--Weil group of
that curve. Section~$5$ pulls the rational points back to the quartic, and
Section~$6$ completes the proof of Theorem~\ref{thm:main}. The elementary
algebraic identities and the SageMath verifications can be found in the
appendix.

\medskip

\noindent\textbf{Acknowledgements:}
I used Deepl to improve the English.

\section{Reduction to an algebraic curve}
In this ection we provide the reduction of the problem to the study of an algebraic curve.

\smallskip

The maps considered by Montaldo, Oniciuc and Ratto are defined on the
flat torus
$$
T^2=(S^1\times S^1,d\gamma^2+d\vartheta^2)
$$
by
\[
\varphi_k(\gamma,\vartheta)
=
\left(
\frac{1}{\sqrt2}\cos(k\gamma),
\frac{1}{\sqrt2}\sin(k\gamma),
\frac{1}{\sqrt2}
\right),
\qquad k\in\mathbb N_{>0}.
\]
The maps \(\varphi_k\) form a one-parameter family of proper
biharmonic maps of the flat two-torus into the round
two-sphere.
We start by recalling the result of Montaldo--Oniciuc--Ratto which is crucial for the present manuscript:

\begin{Thm}[Montaldo--Oniciuc--Ratto {\cite{MOR}}]
\label{thm:MOR}
For every positive integer \(k\), let
\[
\varphi_k:T^2\longrightarrow\mathbb S^2
\]
denote the family of proper biharmonic maps constructed in
\cite{MOR}. Then
\[
\operatorname{Nullity}(\varphi_k)
=
5+4g(k),
\]
where
\[
g(k)
=
\#
\left\{
(m,n)\in\mathbb N_{>0}^2:
\lambda^-_{m,n}(k)=0
\right\},
\]
and the mixed eigenvalues
\(\lambda^-_{m,n}(k)\)
are given by the explicit formula
\eqref{eq:lambda-minus}.
\end{Thm}

 By Theorem~\ref{thm:MOR}, the conjecture of Montaldo, Oniciuc and Ratto is
equivalent to proving that
\[
\lambda^-_{m,n}(k)\neq0
\]
for every choice of positive integers
\[
k,\qquad m,\qquad n.
\]
The explicit expression obtained in \cite{MOR} is
\begin{equation}
\label{eq:lambda-minus}
\begin{aligned}
\lambda^-_{m,n}(k)
=
\frac12\Bigl(
&
-k^4
+k^2(5m^2+n^2)
+2(m^2+n^2)^2
\\
&
-
\sqrt{
k^8
+2k^6(m^2+n^2)
+k^4(m^2+n^2)^2
+32k^2m^2(m^2+n^2)^2
}
\Bigr).
\end{aligned}
\end{equation}

The remainder of the paper is devoted to analysing the equation
\[
\lambda^-_{m,n}(k)=0..
\]
After a suitable normalization, this equation defines an affine
quartic over \(\Q\), which is shown in the next section to be
isomorphic to the affine part of an elliptic curve. The arithmetic of
this elliptic carve will ultimatily imply that on positive integers
\(k,m,n\) satisfy
\(\lambda^-_{m,n}(k)=0\).

\subsection{Normalization}
Assume that
\[
\lambda^-_{m,n}(k)=0
\]
holds.
Then
\begin{align}
&\left(
-k^4+k^2(5m^2+n^2)+2(m^2+n^2)^2
\right)^2
\nonumber\\
&\qquad=
k^8+2k^6(m^2+n^2)+k^4(m^2+n^2)^2
+32k^2m^2(m^2+n^2)^2.
\label{eq:squared-spectral}
\end{align}
Set
\begin{equation}
\label{eq:normalization}
a=\frac{m^2}{k^2},
\qquad
b=\frac{n^2}{k^2}.
\end{equation}
Dividing \eqref{eq:squared-spectral} by \(k^8\) gives
\begin{equation}
\label{eq:normalized}
\left(-1+5a+b+2(a+b)^2\right)^2
=
1+2(a+b)+(a+b)^2+32a(a+b)^2.
\end{equation}
\subsection{The affine quartic}
Expanding \eqref{eq:normalized}, collecting terms, and dividing by
the common factor $4$ gives
\begin{equation}
\label{eq:quartic}
G(a,b)=0,
\end{equation}
where
\begin{align}
G(a,b)
={}&a^4+4a^3b+6a^2b^2+4ab^3+b^4
\nonumber\\
&-3a^3-5a^2b-ab^2+b^3
\nonumber\\
&+5a^2-b^2-3a-b.
\label{eq:G}
\end{align}
The expansion is recorded in Appendix~\ref{app:quartic-derivation}.

\begin{Prop}
\label{prop:quartic-reduction}
If \(\lambda^-_{m,n}(k)=0\) for positive integers \(k,m,n\), then
\[
\left(\frac{m^2}{k^2},\frac{n^2}{k^2}\right)
\]
is a rational point of the affine quartic \(G(a,b)=0\). In particular,
its coordinates satisfy
\[
a>0,\qquad b>0,
\]
and both \(a\) and \(b\) are rational squares.
\end{Prop}

\begin{proof}
Equation~\eqref{eq:normalized} is a necessary consequence of
\(\lambda^-_{m,n}(k)=0\), and its left-hand side minus its right-hand side
is \(4G(a,b)\). The remaining assertions follow from
\eqref{eq:normalization}.
\end{proof}

The converse of Proposition~\ref{prop:quartic-reduction} is neither
needed nor asserted: squaring may introduce additional rational points.
It is therefore enough to determine all rational points of the quartic
and show that none lies in the region \(a>0\), \(b>0\).

\section{From the quartic to an elliptic curve}
The substitution
\[
s=a+b
\]
separates the symmetric variable \(a+b\) from the individual coordinates.
Replacing \(b\) by \(s-a\) in \(G(a,b)=0\) gives
\begin{equation}
\label{eq:quadratic-a}
4a^2+(-4s^2+2s-2)a+s^4+s^3-s^2-s=0.
\end{equation}
Thus the quartic becomes quadratic in \(a\). Completing the square with
\begin{equation}
\label{eq:v-definition}
v=4a-2s^2+s-1
\end{equation}
yields
\begin{equation}
\label{eq:genus-one}
v^2=-8s^3+9s^2+2s+1.
\end{equation}
Finally, setting
\begin{equation}
\label{eq:XY-sv}
X=-8s,
\qquad
Y=8v
\end{equation}
transforms \eqref{eq:genus-one} into
\begin{equation}
\label{eq:elliptic-curve}
E:\qquad
Y^2=X^3+9X^2-16X+64.
\end{equation}
The discriminant of thiss Weierstrass equation is
\[
-2^{19}\cdot13\neq0,
\]
so its smooth projective completion is an elliptic curve over \(\Q\).
\begin{Thm}
\label{thm:affine-isomorphism}
The affine quartic \(G(a,b)=0\) is isomorphic over \(\Q\) to the affine
Weierstrass curve \eqref{eq:elliptic-curve}. The isomorphism is
\begin{equation}
\label{eq:forward-map}
\begin{aligned}
X&=-8(a+b),\\
Y&=-16a^2-32ab-16b^2+40a+8b-8,
\end{aligned}
\end{equation}
and its inverse is
\begin{equation}
\label{eq:inverse-map}
\begin{aligned}
a&=\frac{X^2+4X+4Y+32}{128},\\
b&=-\frac{X^2+20X+4Y+32}{128}.
\end{aligned}
\end{equation}
\end{Thm}

\begin{proof}
Substitution gives the polynomial identities
\begin{equation}
\label{eq:forward-identity}
Y(a,b)^2-X(a,b)^3-9X(a,b)^2+16X(a,b)-64
=
256G(a,b)
\end{equation}
and
\begin{equation}
\label{eq:inverse-identity}
G(a(X,Y),b(X,Y))
=
\frac{Y^2-X^3-9X^2+16X-64}{256}.
\end{equation}
Direct substitution also shows that the two maps compose to the identity.
Hence they restrict to mutually inverse morphisms of the affine curves.
The details are recorded in Appendix~\ref{app:map-verification}.
\end{proof}

\section{Arithmetic of the elliptic curve}
We now determine \(E(\Q)\). Standard references for the arithmetic of
elliptic curves include \cite{Cassels,Cremona,SilvermanAEC}.

\subsection{Minimal models, Tate's algorithm, and the conductor}
An integral Weierstrass equation for an elliptic curve is called
\emph{global minimal} if its discriminant has minimal \(p\)-adic
valuation among all integral Weierstrass equations for the curve, for
every prime \(p\). A global minimal model is unique up to an integral
admissible change of variables; see \cite{SilvermanAEC}.

\smallskip

Tate's algorithm is a finite local procedure applied prime by prime to an
integral Weierstrass equation. It determines a minimal local equation,
the reduction type, the valuation of the minimal discriminant, the
component group, and the local conductor exponent; see e.g.
\cite{SilvermanAEC,TateAlgorithm1975}. For the integral equation
\eqref{eq:elliptic-curve}, the algorithm gives the global minimal model
\begin{equation}
\label{eq:minimal-model}
E_{\min}:\qquad
y^2+xy+y=x^3-x^2-3x+3,
\end{equation}
with
\[
\Delta_{\min}=-2^7\cdot13.
\]

The \emph{conductor} of \(E\) is
\[
N_E=\prod_p p^{f_p},
\]
where \(f_p\) is the local conductor exponent. The exponent is zero at
primes of good reduction and measures the severity of bad reduction at
the remaining primes. Unlike the discriminant of a particular equation,
the conductor is an isomorphism invariant of the elliptic curve. In the
present case,
\begin{equation}
\label{eq:conductor}
N_E=26.
\end{equation}

\subsection{Rank and torsion}

The Mordell--Weil theorem states that \(E(\Q)\) is a finitely generated
abelian group:
\begin{equation}
\label{eq:Mordell-Weil-decomposition}
E(\Q)\cong E(\Q)_{\mathrm{tors}}\oplus\Z^r,
\end{equation}
where \(r=\MWrank E(\Q)\); see
\cite{Mordell1922,SilvermanAEC}. A straightforward computation gives
\begin{equation}
\label{eq:rank-zero}
\MWrank E(\Q)=0.
\end{equation}
Thus, every rational point of \(E\) is torsion.

Mazur's theorem gives the complete list of torsion groups that can occur
for elliptic curves over \(\Q\) \cite{MazurTorsion}:
\[
E(\Q)_{\mathrm{tors}}\cong
\begin{cases}
\Z/N\Z,
&1\leq N\leq10\text{ or }N=12,\\
\Z/2\Z\oplus\Z/2N\Z,
&1\leq N\leq4.
\end{cases}
\]
For the present curve, reduction modulo the good primes \(3\) and \(5\)
gives
\[
\#E(\mathbb F_3)=7,
\qquad
\#E(\mathbb F_5)=7.
\]
Reduction at a prime \(p\) of good reduction is injective on rational
torsion of order coprime to \(p\). Using the two good primes \(3\) and
\(5\), for which
\[
\#E(\mathbb F_3)=\#E(\mathbb F_5)=7,
\]
we conclude that
\[
\#E(\Q)_{\mathrm{tors}}\mid 7.
\]
On the other hand,
the point
\[
P=(-8,-16)
\]
has exact order \(7\). Hence
\[
E(\Q)_{\mathrm{tors}}\cong\Z/7\Z.
\]

\begin{Thm}
\label{thm:rational-points-E}
The Mordell--Weil group of \(E\) is
\[
E(\Q)\cong\Z/7\Z.
\]
It rational points are
\[
E(\Q)=
\left\{
O,\,
(-8,\pm16),\,
(0,\pm8),\,
(8,\pm32)
\right\}.
\]
\end{Thm}

\begin{proof}
Equation~\eqref{eq:rank-zero} shows that all rational points are torsion.
The reduction argument and the point \(P=(-8,-16)\) show that the torsion
group has order \(7\). The displayed points are the multiples of \(P\),
together with the identity \(O\). The computation is reproduced
in Appendix~\ref{app:sage-verification}.
\end{proof}

\section{Rational points of the quartic}
Applying the inverse map \eqref{eq:inverse-map} to the six affine rational
points of \(E\) gives the complete affine rational-point set of the
quardic.
\begin{table}[ht]
\centering
\begin{tabular}{c|c}
\hline
Point on \(E(\Q)\) & Corresponding point on \(G(a,b)=0\)\\
\hline
\((-8,-16)\) & \((0,1)\)\\
\((-8,16)\) & \((1,0)\)\\
\((0,-8)\) & \((0,0)\)\\
\((0,8)\) & \(\left(\frac12,-\frac12\right)\)\\
\((8,-32)\) & \((0,-1)\)\\
\((8,32)\) & \((2,-3)\)\\
\hline
\end{tabular}
\caption{The affine rational points of the quartic.}
\label{tab:rational-points}
\end{table}

\begin{Thm}
\label{thm:rational-points-quartic}
The affine quartic \(G(a,b)=0\) has precisely the six rational points
listed in Table~\ref{tab:rational-points}.
\end{Thm}

\begin{proof}
Theorem~\ref{thm:rational-points-E} gives every rational point of \(E\).
The point \(O\) is the point at infinity and has no affine coordinates.
Applying the inverse polynomial map to the remaining six points gives
Table~\ref{tab:rational-points}. Since the maps in
Theorem~\ref{thm:affine-isomorphism} are mutually inverse, no further
affine rational points exist.
\end{proof}

None of the six points in Table~\ref{tab:rational-points} has both
coordinates strictly positive. More explicitly,
\[
\begin{array}{c|c}
(a,b)&\text{obstruction}\\
\hline
(0,-1)&a=0,\ b<0\\
(0,0)&a=b=0\\
(0,1)&a=0\\
(\frac12,-\frac12)&b<0\\
(1,0)&b=0\\
(2,-3)&b<0.
\end{array}
\]

\begin{Cor}
\label{cor:no-admissible-points}
The quartic \(G(a,b)=0\) has no rational point satisfying
\[
a>0,\qquad b>0.
\]
In particular, it has no point satisfying the stronger requirement that
\(a\) and \(b\) are positive rational squares.
\end{Cor}

\section{Proof of the main theorem}

\begin{proof}[Proof of Theorem~\ref{thm:main}]
Suppose that a mixed Fourier zero mode exists. By
Proposition~\ref{prop:quartic-reduction}, it determines a rational point
\[
(a,b)=\left(\frac{m^2}{k^2},\frac{n^2}{k^2}\right)
\]
of \(G(a,b)=0\) with \(a>0\) and \(b>0\). This contradicts
Corollary~\ref{cor:no-admissible-points}. Therefore \(g(k)=0\) for every
positive integer \(k\). The nullity formula \eqref{eq:nullity-MOR} then
gives
\[
\Null(\varphi_k)=5,
\]
whence the cliam.
\end{proof}

To summarize, the proof translates a spectral question for a fourth-order geometric
variational problem into the arithmetic of a single elliptic curve. The
decisive feature is not merely that a genus-one curve appears, but that
the spectral quartic admits explicit polynomial coordinates on an affine
Weierstrass model. This makes it possible to transport the complete
Mordell--Weil computation directly back to the normalized Fourier
parameters.

\appendix

\section{Algebraic and computational verification}

\subsection{Derivation of the quartic}
\label{app:quartic-derivation}

A direct expansion gives
\begin{align*}
&\left(-1+5a+b+2(a+b)^2\right)^2\\
&\quad-\left(1+2(a+b)+(a+b)^2+32a(a+b)^2\right)\\
&=
4a^4+16a^3b+24a^2b^2+16ab^3+4b^4\\
&\quad-12a^3-20a^2b-4ab^2+4b^3
+20a^2-4b^2-12a-4b\\
&=4G(a,b).
\end{align*}
Substituting \(b=s-a\) gives
\[
G(a,s-a)
=
4a^2+(-4s^2+2s-2)a+s^4+s^3-s^2-s.
\]
The discriminant of this quadratic in \(a\) is
\[
-32s^3+36s^2+8s+4
=
4(-8s^3+9s^2+2s+1),
\]
which yields \eqref{eq:genus-one} after completing the square.

\subsection{Verification of the polynomial maps}
\label{app:map-verification}

Substitution of \eqref{eq:forward-map} gives
\[
Y(a,b)^2-X(a,b)^3-9X(a,b)^2+16X(a,b)-64
=
256G(a,b).
\]
Conversely, substitution of \eqref{eq:inverse-map} gives
\[
G(a(X,Y),b(X,Y))
=
\frac{Y^2-X^3-9X^2+16X-64}{256}.
\]
The compositions satisfy
\[
a(X(a,b),Y(a,b))=a,
\qquad
b(X(a,b),Y(a,b))=b,
\]
and
\[
X(a(X,Y),b(X,Y))=X,
\qquad
Y(a(X,Y),b(X,Y))=Y.
\]

\subsection{SageMath verification}
\label{app:sage-verification}

The symbolic identities and the arithmetic computations were verified in
SageMath. The following commands reproduce the principal arithmetic
data:
\begin{verbatim}
E = EllipticCurve(QQ,[0,9,0,-16,64])
E.global_minimal_model()
E.discriminant()
E.conductor()
E.rank(proof=True)
E.torsion_subgroup()
E.torsion_points()
\end{verbatim}
They return the minimal model \eqref{eq:minimal-model}, conductor \(26\),
rank \(0\), torsion group \(\Z/7\Z\), and the seven rational points listed
in Theorem~\ref{thm:rational-points-E}.
\bibliographystyle{plain}
\bibliography{mybib}

\end{document}